\documentclass[12pt, a4paper]{amsart}
\usepackage{graphicx} % Required for inserting images
\usepackage{amsthm}
\usepackage{amssymb}
\usepackage{amsmath}
\usepackage{placeins}
\usepackage{float} % in preamble
\usepackage{tikz}
\usepackage[colorlinks=true, allcolors=blue]{hyperref}
\newcounter{theorem}
\numberwithin{theorem}{section}
\numberwithin{equation}{section}
\newtheorem{thm}[theorem]{Theorem}
\newtheorem*{thm*}{Theorem \thesubsection}
\newtheorem{lemma}[theorem]{Lemma}
\newtheorem{prop}[theorem]{Proposition}
\newtheorem{cor}[theorem]{Corollary}
\newtheorem*{lemma*}{Lemma \thesubsection}
\newtheorem*{prop*}{Proposition \thesubsection}
\newtheorem*{cor*}{Corollary \thesubsection}
\theoremstyle{definition}
\newtheorem{definition}[theorem]{Definition}
\newtheorem*{definition*}{Definition \thesubsection}

\newtheorem*{example*}{Example \thesubsection}
\theoremstyle{remark}

\newtheorem*{remark*}{Remark \thesubsection}

\hypersetup{
  pdftitle={BINDING NUMBERS OF TIGHT CONTACT STRUCTURES ON
L(n, 1)},
  pdfauthor={Csaba Daniel Farkas}
}
\def\sideremark#1{\ifvmode\leavevmode\fi\vadjust{\vbox to0pt{\vss% the remark
 \hbox to 0pt{\hskip\hsize\hskip1em%                          will appear only
 \vbox{\hsize3cm\tiny\raggedright\pretolerance10000%          on the side
  \noindent #1\hfill}\hss}\vbox to8pt{\vfil}\vss}}}%
\title{Binding Numbers of Tight Contact Structures on $L(n,1)$}
\author{Csaba Daniel Farkas}

\begin{document}

\maketitle

\begin{abstract}
    We study binding numbers of tight contact structures on the lens spaces $L(n,1)$. Using the $d_3-$invariant, together with restrictions on planar monodromy factorizations and the Durst--Kegel algorithm for computing $d_3$ from open books, we obtain lower bounds for the number of binding components of planar open books supporting these contact structures. As a consequence, we compute the binding number of the universally tight contact structures on $L(n,1)$, showing that it is equal to $n$.
\end{abstract}

\section{Introduction}
Giroux \cite{Giroux02} proved that all contact structures on closed oriented $3$-manifolds are obtained from open book decompositions. This enabled the study of contact structures through the open books supporting them, thus establishing open book decompositions as highly relevant in investigating contact structures. In this vein, Etnyre--Ozbagci \cite{etnyre2006invariantscontactstructuresopen} define invariants such as support genus, support norm and binding number of contact structures. 

The support genus of a contact structure is the smallest genus of a page of
a supporting open book. Once this genus is fixed, the binding number records
the smallest possible number of boundary components of such a page. Thus, for
a planar contact structure, the binding number is the smallest number of
binding components among all planar open books supporting it.
Although planarity is known in many situations---for instance, every overtwisted
contact structure is planar \cite{EtnyrePlanar}---exact computations of binding
numbers for tight contact structures remain comparatively scarce.

\begin{definition}
    For $n>1$, we denote by $(L(n,1),\xi^n_\rho)$ the tight contact $3-$manifold, where $\xi^n_\rho$ is the tight contact structure obtained by performing a contact $(-1)$-surgery on a $\operatorname{tb}=-(n-1)$ Legendrian unknot in $(S^3,\xi_{\text{std}})$ with $\operatorname{rot}=\rho$.
\end{definition}

In this paper, we study the binding numbers of the tight contact structures
$\xi_\rho^n$ on $L(n,1)$. Our main result gives a lower bound depending on
$|\rho|$. The proof begins with a hypothetical planar supporting open book
with $d$ binding components. Results of Wendl \cite{Wendl_2010} and
Plamenevskaya--Van Horn-Morris \cite{Plamenevskaya_2010} constrain the possible positive
factorizations of its monodromy, while the Durst--Kegel algorithm computes the
$d_3$--invariant from the corresponding open-book data. Comparing this value
with the $d_3$--invariant of $\xi_\rho^n$ yields a lower bound for $d$.

Our main theorems are the following.

\begin{thm}\label{thm::bn}
    For $n>1$ with $n\neq 4$, the tight contact structures $\xi^n_\rho$ on $L(n,1)$ have $$\operatorname{bn}(\xi^n_\rho)\geq \left\lceil 1+\frac{(1+|\rho|)^2}{n-1}\right\rceil.$$
\end{thm}

This lower bound is weak in general, but in several cases it is strong enough to determine the binding number exactly.

As an immediate corollary, we obtain:
\begin{cor}\label{cor::untight}
    The universally tight contact structures on $L(n,1)$ for $n\neq4$ have binding number equal to $n$.
\end{cor}
\begin{proof}
    The universally tight contact structures are exactly the ones where $\rho=\pm (n-2)$. Theorem \ref{thm::bn} yields the lower bound $\operatorname{bn}(\xi^n_{\pm (n-2)})\geq n$. The reverse inequality follows from the construction in Section \ref{sec::upper}.
\end{proof}

Furthermore, with methods developed in the proof of Theorem  \ref{thm::bn}, we also prove the following.

\begin{thm}\label{thm::bn56}
    For $n\leq 6$, every tight contact structure on \(L(n,1)\) has binding number \(n\). Furthermore, for $n\geq 7$, the following exact values hold:
\[
        \operatorname{bn}(\xi^n_\rho)=
        \begin{cases}
        n, & \text{for } |\rho|=n-4,n-2\\
        n-3, & \text{for } |\rho|=n-6.
        \end{cases}
\]
\end{thm}
The exceptional case $\operatorname{bn}(\xi^4_0)=4$ was previously obtained by Etnyre--Ozbagci \cite{etnyre2006invariantscontactstructuresopen}.

We begin by recalling the necessary background on open books, Legendrian
knot invariants, and tight contact structures on lens spaces. We then 
construct supporting open books for these tight lens spaces and finish with a proof of the lower bounds.

\section{Preliminaries}\label{sec::prelim}

In this section we recall various facts about open book decompositions, we revisit the algorithm of Durst--Kegel \cite{durst2019computingrotationnumbersopen} to compute the Thurston--Bennequin and rotation number of a Legendrian knot in the page of an open book decomposition, we then turn to the classification and construction of tight contact structures on the lens spaces $L(n,1)$ \cite{HondaClassification}.

\subsection{Open books in contact topology}

We recall the standard conventions for open books and supported contact structures. For a detailed introduction we refer the reader to \cite{Geiges08} or \cite{OS04}. 

An abstract open book is a pair \((\Sigma,\varphi)\), where \(\Sigma\) is a
compact oriented surface with non-empty boundary and
\(\varphi \in \operatorname{Diff}^+(\Sigma,\partial\Sigma)\) is an
orientation-preserving diffeomorphism which is the identity near
\(\partial\Sigma\), considered up to isotopy relative to \(\partial\Sigma\).
The map \(\varphi\) is called the monodromy.

The associated closed oriented 3-manifold $M_{(\Sigma,\varphi)}$ is constructed as follows. First form
the mapping torus
\[
T_\varphi =
\frac{[0,1]\times \Sigma}{(1,x)\sim(0,\varphi(x))}.
\]
For each boundary component \(C\subset \partial\Sigma\), the boundary of
\(T_\varphi\) contains a torus \(S^1\times C\). We attach a solid torus
\(D^2\times C\) by identifying $\partial D^2\times C$ with $S^1\times C$. The union of the core circles \(\{0\}\times C\) of these attached solid tori is called the binding. The closures of the fibers of the natural map to \(S^1\) are called the pages.

A positive cooriented contact structure \(\xi\) on $M_{(\Sigma,\varphi)}$ is said to be supported by the abstract open book \((\Sigma,\varphi)\) if there exists a contact form \(\alpha\) for \(\xi\) such that \(\alpha\) is positive on the oriented binding components \(B\), and \(d\alpha\) restricts to a positive area form on every page. 

By the Thurston--Winkelnkemper construction \cite{TW75}, every abstract open book
\((\Sigma,\varphi)\) supports a positive cooriented contact structure on the
associated 3-manifold \(M_{(\Sigma,\varphi)}\). Giroux later showed that
this contact structure is unique up to isotopy, and conversely that every
positive cooriented contact structure on a closed oriented 3-manifold is
supported by some open book \cite{Giroux02}. More generally, we say that an abstract open book \((\Sigma,\varphi)\) supports a contact manifold \((M,\xi)\) if there is an orientation-preserving diffeomorphism $M_{(\Sigma,\varphi)} \cong M$
under which a contact structure supported by \((\Sigma,\varphi)\) is isotopic
to \(\xi\).

In order to study contact structures through their supporting open books, one may define the following invariants.
\begin{definition}\cite{etnyre2006invariantscontactstructuresopen}
The support genus \(\operatorname{sg}(\xi)\) of a contact structure \(\xi\) is
the minimal genus of a page among all open books supporting \(\xi\). The binding
number \(\operatorname{bn}(\xi)\) is the minimal number of binding components
among all supporting open books whose page genus equals \(\operatorname{sg}(\xi)\).
\end{definition}

\subsection{The Thurston-Bennequin Invariant and Rotation numbers in open books}\label{sec:invariants}

We recall the algorithm of Durst--Kegel \cite{durst2019computingrotationnumbersopen} for computing the Thurston-Bennequin invariant and the rotation numbers of Legendrian knots presented on pages of abstract open books. We also recall the computation of the $d_3$-invariant in this setting. In our setting, the algorithm simplifies considerably, since we only consider planar open books whose monodromy admits a factorization into positive Dehn twists.

\begin{definition} \label{def::basis}
Let $\Sigma$ be a planar surface with $d$ boundary components. We distinguish one boundary component as the outer boundary component and label the remaining $d-1$ boundary components by $1,\ldots,d-1$. Let $\beta_i$ be the positively oriented simple closed curve enclosing precisely the $i$-th inner boundary component; see Figure \ref{fig::stdOB}. Let $b_i$ be a properly embedded arc connecting the $i$-th inner boundary component to the outer boundary component.

We choose the orientations so that the arcs $b_i$ are pairwise disjoint, the curves $\beta_i$ are pairwise disjoint, and $\beta_i\cdot b_j=\delta_{ij}$.

Let $K\subset\Sigma$ be an oriented embedded closed curve which has been isotoped so that it intersects the arcs $b_i$ minimally. Reading $K$ once in its orientation, we form a cyclic word $w=\beta_{i_1}^{\varepsilon_1}\cdots \beta_{i_k}^{\varepsilon_k}$ in the letters $\langle \beta_1,\cdots, \beta_{d-1} \rangle$, where we write a $\beta_i$ at each positive intersection of $K$ with $b_i$ and a $\beta_i^{-1}$ at each negative intersection. Since \(w\) is a cyclic word, we set \(i_{k+1}=i_1\).
     Denote the places in the word $w$ where the index changes by $r_u$ $(r_d)$ if the index increases (decreases), where the transition from the last letter to the first letter is included. More precisely for each \(j\in\{1,\cdots,k\}\), define
\[
\alpha_j =
\begin{cases}
r_u, & i_{j+1 }>i_{j },\\
r_d, & i_{j+1 }<i_{j }.
\end{cases}
\]
     Then define $r(K)=D-U$ where $$D:=\#\{j\in\{1,\cdots,k\}: \varepsilon_j=+1 \text{ and } \alpha_j=r_d\}$$ is the number of places where a $\beta_i$ is followed by $r_d$ and $$U:=\#\{j\in\{1,\cdots,k\}: \varepsilon_j=-1 \text{ and } \alpha_j=r_u\}$$ is the number of places where a $\beta_i^{-1}$ is followed by $r_u$.
\end{definition}

\begin{figure}[h]
    \centering
    \includegraphics[width=0.8\textwidth]{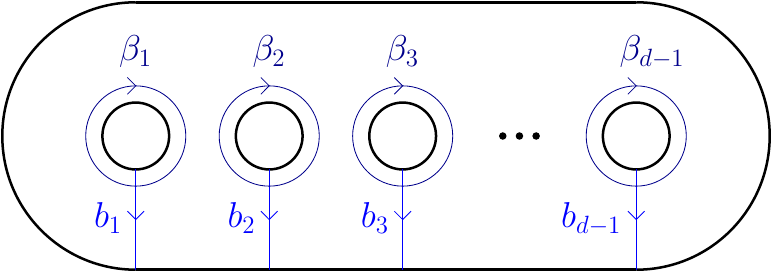}
    \caption{The open book $(\Sigma,D^+_{\beta_{d-1}}\circ\cdots\circ D^+_{\beta_1})$ supports $(S^3,\xi_{\mathrm{std}})$.}
    \label{fig::stdOB}
\end{figure}

Now let $(\Sigma,\varphi)$ be a planar open book with $d$ boundary components and a distinguished outer boundary component and bases $(\beta_i)_i,(b_i)_i$ for $H_1(\Sigma)$ and $H_1(\Sigma,\partial \Sigma )$ as in Definition \ref{def::basis}. Consider a monodromy $\varphi$ which admits a factorization $\varphi=D^+_{\alpha_{k}}\circ\cdots\circ D^+_{\alpha_1}$. Here each $\alpha_j$ is a simple closed curve in $\Sigma$, oriented as the positively oriented boundary of the component not containing the distinguished outer boundary. We associate to this factorization the integer matrix $A=(A_{\ell j})\in M_{(d-1)\times k}(\mathbb{Z})$, where $A_{\ell j}=\alpha_j\cdot b_\ell$ and hence
$$
[\alpha_j]=\sum_{\ell=1}^{d-1}A_{\ell j}[\beta_\ell]\in H_1(\Sigma;\mathbb{Z}).
$$
Equivalently, in the standard planar situation considered here, $A_{\ell j}=1$ if $\alpha_j$ encloses the $\ell$-th inner boundary component, and $A_{\ell j}=0$ otherwise. Following Durst--Kegel \cite{durst2019computingrotationnumbersopen}, define $Q\in M_{(d-1+k)\times (d-1+k)}$ by
$$
Q=\begin{pmatrix}
0_{(d-1)\times(d-1)}&-A\\
-A^T&-A^TA-I_k
\end{pmatrix}.
$$

Let $K\subset\Sigma$ be an embedded closed curve which is Legendrian realized on the page $\Sigma\times\{0\}$ in the contact manifold supported by $(\Sigma,\varphi)$. Denote by $\mathbf{l},\mathbf{r}\in \mathbb{Z}^{d-1+k}$ the vectors with 

\begin{align*}
    \mathbf{l}_j&=-(K\cdot b_j) \qquad \text{for } j\in \{1,\cdots, d-1\},\\
    \mathbf{l}_{d-1+j}&=-\sum_{i=1}^{d-1}(K\cdot b_i)(\alpha_j\cdot b_i) \qquad \text{for } j\in \{1,\cdots, k\},\\
    \mathbf{r}_j&=0 \qquad \text{for } j\in \{1,\cdots, d-1\},\\
    \mathbf{r}_{d-1+j}&=r(\alpha_j) \qquad \text{for } j\in \{1,\cdots, k\}.
\end{align*}

Now one considers the equations $Q\mathbf{a}=\mathbf{l}$ and $Q\mathbf{b}=\mathbf{r}$ for $\mathbf{a} \in \mathbb{Z}^{d-1+k},\mathbf{b}\in \mathbb{Q}^{d-1+k}$. The existence of an integral solution for $\mathbf{a}$ is equivalent to $K$ being nullhomologous. The existence of a solution for $\mathbf{b}$ is equivalent to $e(\xi)$ being torsion.
Then the invariants $\operatorname{tb}(K)$, $\operatorname{rot}(K)$, and $d_3(\xi)$ are computed via
\begin{align*}
    \operatorname{tb}(K)&=-\sum_{i=1}^{d-1}(K\cdot b_i)^2-\langle \mathbf{a},\mathbf{l}\rangle,\\
    \operatorname{rot}(K)&=r(K)-\langle \mathbf{a},\mathbf{r}\rangle,\\
    d_3(\xi)&=-1+\frac{d-k}{2}+\frac{1}{4}\langle \mathbf{b},\mathbf{r}\rangle - \frac{3}{4}\sigma(Q).
\end{align*}

Throughout, we use the convention for the $d_3$-invariant in which the standard tight contact structure on $S^3$ satisfies
$
d_3(\xi_{\mathrm{std}})=-\frac{1}{2}.$
With this convention, Gompf's \cite{gompf1998handlebodyconstructionsteinsurfaces} invariant $\theta$ satisfies $\theta(\xi)=4d_3(\xi)$.

The following will be important in the proof of Theorem \ref{thm::bn}.

\begin{lemma}\label{lemma::r}
Let $\Sigma$ be a planar surface with $d$ boundary components, with one boundary component distinguished as the outer boundary component. Choose an arc basis as in Figure \ref{fig::stdOB}. Let $K\subset\Sigma$ be an essential positively oriented simple closed curve with
$$
[K]=\sum_{i=1}^{d-1}\varepsilon_i[\beta_i],
$$
where $\varepsilon_i\in\{0,1\}$. Then
$$
r(K)=\sum_{i=1}^{d-1}\varepsilon_i-1.
$$
\end{lemma}

\begin{proof}

Consider the auxiliary open book
$
(\Sigma,D^+_{\beta_{d-1}}\circ\cdots\circ D^+_{\beta_1}),
$
which is obtained by plumbing $d-1$ positive Hopf bands and which supports $(S^3,\xi_{\mathrm{std}})$. In this open book, a positively oriented curve enclosing $e=\sum_{i=1}^{d-1}\varepsilon_i$ inner boundary components is a Legendrian unknot with $\operatorname{tb}=-e$ and $\operatorname{rot}=e-1$, see Proposition \ref{curveStab}. 

For this auxiliary monodromy, the matrix $A$ is the identity matrix $I_{d-1}$. Moreover, each $\beta_i$ encloses exactly one inner boundary component, and hence $r(\beta_i)=0$. Applying the Durst--Kegel rotation formula to $K$, we get
$$
e-1=\operatorname{rot}(K)=r(K)-\langle \mathbf{a},\mathbf{r}\rangle=r(K),
$$
where the last equality holds since $\mathbf{r}=0$. Notice that an integral solution to $\mathbf{a}$ always exists since every knot in $S^3$ is nullhomologous.
This proves the claim.
\end{proof}

\subsection{Tight $L(n,1)$ lens spaces}\label{sec::lensspaces}

Consider a Legendrian unknot $K$ in $(S^3,\xi_{\mathrm{std}})$ with $\operatorname{tb}(K)=-(n-1)$ and $\operatorname{rot}(K)=r$. Topologically, contact $(-1)$-surgery on $K$ is Dehn surgery with
coefficient $\operatorname{tb}(K)-1=-n$ relative to the Seifert framing.
Thus the resulting smooth manifold is $L(n,1)$. For an introduction to
contact surgery we refer to \cite{Geiges08}.

Since contact $(-1)$-surgery preserves Stein fillability, and since $(S^3,\xi_{\mathrm{std}})$ is Stein fillable, the resulting contact
structure is Stein fillable, and in particular tight. Honda's classification \cite{HondaClassification} shows that every tight contact structure on $L(n,1)$ arises in this way. The tight contact structures are distinguished by the possible rotation numbers $\operatorname{rot}(K)\in\{-(n-2),-(n-4),\dots,n-4,n-2\}.$
Thus, for $n>1$, there are $n-1$ tight contact structures on $L(n,1)$.
One calls the contact structures corresponding to $\operatorname{rot}(K)=\pm (n-2)$ universally tight, otherwise we call it virtually overtwisted. This terminology reflects the fact that the lift to the universal cover \(S^{3}\) is tight in the universally tight case and overtwisted in the virtually overtwisted case.

It is known that a contact manifold defined via a contact $(-1)$--surgery on a Legendrian knot up to contactomorphism is independent of the orientation of the Legendrian knot. It follows that the two contact structures $\xi^n_\rho$ and $\xi^n_{-\rho}$ are contactomorphic.

The following filling property will be crucial in the proof of the main theorem.

\begin{thm}[\cite{Plamenevskaya_2010,McD90}]
\label{thm::plam}
If $n\neq 4$, or if $(n,\rho)=(4,0)$, then the tight contact structure $\xi^n_\rho$ on $L(n,1)$ has a unique Stein filling up to symplectic deformation. In the cases $(n,\rho)=(4,\pm 2)$ it has exactly two Stein fillings.
\end{thm}

We shall combine this with Wendl's theorem on planar open books.

\begin{thm}[Wendl, \cite{Wendl_2010}]
\label{thm::wendl}
If $(M,\xi)$ is a planar contact manifold, then it is strongly (and thus Stein) fillable if and only if every supporting planar open book has monodromy isotopic to a product of positive Dehn twists.
\end{thm}

These two results imply the following.

\begin{lemma}
\label{lemma:factorization}
Let $(\Sigma,\varphi)$ be a planar open book with $d$ boundary components supporting the tight contact structure $\xi^n_\rho$ on $L(n,1)$. Then $\varphi$ admits a positive factorization consisting of exactly $d$ right-handed Dehn twists if $n\neq4$ or $(n,\rho)=(4,0)$. In the case of $(n,\rho)=(4,\pm 2)$ there is a positive factorization consisting of exactly $d-1$ or $d$ right-handed Dehn twists.
\end{lemma}

\begin{proof}
As explained above, every tight contact structure on $L(n,1)$ is Stein fillable. By Theorem \ref{thm::wendl}, the monodromy $\varphi$ admits a factorization into right-handed Dehn twists. Suppose this positive factorization contains $k$ Dehn twists.

The open book $(\Sigma,\operatorname{id})$ gives the contact manifold $\#_{d-1}S^1\times S^2$, whose standard Stein filling is $\natural_{d-1}S^1\times D^3.$ This filling is built from one Stein $0$-handle and $d-1$ Stein $1$-handles. Each right-handed Dehn twist in a positive factorization of $\varphi$ corresponds to attaching a Stein $2$-handle. Therefore the Stein filling associated to this factorization has Euler characteristic $\chi(X)=1-(d-1)+k=2-d+k.$

On the other hand, the tight contact structure on $L(n,1)$ under consideration is obtained by contact $(-1)$-surgery on a Legendrian unknot in $(S^3,\xi_{\mathrm{std}})$. The corresponding Stein filling is obtained from $D^4$ by attaching one Stein $2$-handle, and therefore has Euler characteristic $\chi=1+1=2.$ Theorem \ref{thm::plam} implies that the two fillings have the same Euler characteristic in the case of $n\neq 4$ or $(n,\rho)=(4,0)$. In particular, $2=2-d+k$, which proves our claim in this case.

In the case of $(n,\rho)=(4,\pm 2)$ there are two Stein fillings. 
We now compute their Euler characteristics. In Section \ref{sec::upper} we construct supporting abstract open books for $(L(4,1),\xi^4_{\pm 2})$ as in Figure \ref{fig::OBl41}. The monodromy of the abstract open book is the product of four right-handed Dehn twists, which yields a Stein filling with $\chi=2$; one may factorize it into $3$ right-handed Dehn twist curves using the lantern relation, which yields the second filling with $\chi=1.$

\begin{figure}[h]
    \centering
    \includegraphics[width=0.95\textwidth]{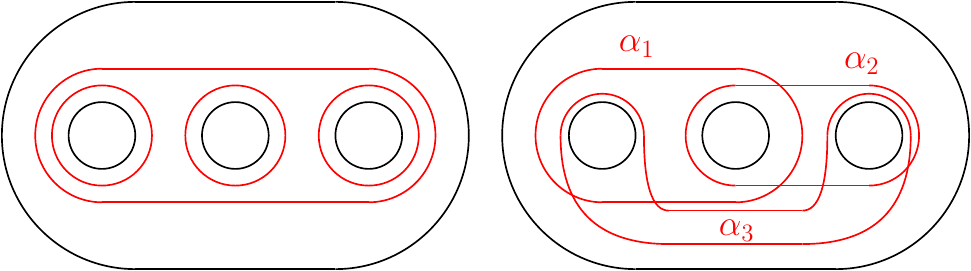}
    \caption{    Two positive factorizations of the same monodromy supporting the universally tight contact structures $\xi^4_{\pm 2}$ on $L(4,1)$. The factorizations are related by the lantern relation. The right diagram has monodromy $D_{\alpha_3}^+\circ D_{\alpha_2}^+ \circ D_{\alpha_1}^+$.}
    \label{fig::OBl41}
\end{figure}

\end{proof}

Finally, the above Legendrian surgery description gives the $d_3$-invariant.
Indeed, let $X$ be the Stein filling of $(L(n,1),\xi^n_\rho)$ obtained from $D^4$ by attaching one
Stein 2-handle along a Legendrian unknot $K$ with
$\operatorname{tb}(K)=-(n-1)$ and $\operatorname{rot}(K)=\rho$. The topological framing is therefore $\operatorname{tb}(K)-1=-n$, and hence $$Q_X=[-n], \qquad \chi(X)=2, \qquad \sigma(X)=-1.$$ Furthermore, Gompf's formula gives \[
\langle c_1(X,J),[S]\rangle=\rho,
\]
where $[S]$ is the generator of $H_2(X)$ represented by the core of the
2-handle capped off by a Seifert surface for $K$. See \cite{gompf1998handlebodyconstructionsteinsurfaces} for more details. It follows then that $$c_1(X,J)^2=\rho\cdot Q_X^{-1}\cdot \rho=-\frac{\rho^2}{n}$$ and $$4d_3(\xi)
=
c_1(X,J)^2-2\chi(X)-3\sigma(X)
=
-\frac{\rho^2}{n}-4+3
=
-\frac{\rho^2}{n}-1.$$

\section{Standard open books and the upper bound}\label{sec::upper}

In this section we prove the following Lemma.
\begin{lemma}\label{lemma::upper}
    The tight contact structure $\xi^n_\rho$ on $L(n,1)$ has binding number $$\operatorname{bn}(\xi^n_\rho)\leq n.$$
\end{lemma}

We first recall two standard facts about open books which will be used
throughout the section; see \cite[Lemma~3.3]{EtnyrePlanar} for the stabilization
statement and \cite[Theorem~5.7]{EtnyreOpenBooks} for the surgery statement.

\begin{prop}\label{curveStab}
Let $(\Sigma,\varphi)$ be an abstract open book and let
$\gamma\subset \Sigma$ be a simple closed curve. Then there exists a
positive stabilization $(\Sigma',\varphi')$ of $(\Sigma,\varphi)$ and a
simple closed curve $\gamma'\subset \Sigma'$ such that the Legendrian
realization $K_{\gamma'}$ in
$(M_{(\Sigma',\varphi')},\xi_{(\Sigma',\varphi')})$ is Legendrian isotopic to
either the positive or the negative Legendrian stabilization of
$K_\gamma$ in $(M_{(\Sigma,\varphi)},\xi_{(\Sigma,\varphi)})$, after identifying $(M_{(\Sigma',\varphi')},\xi_{(\Sigma',\varphi')})$ with $(M_{(\Sigma,\varphi)},\xi_{(\Sigma,\varphi)})$ by the contactomorphism induced by positive stabilization. See Figure \ref{fig::stab}.
\end{prop}

\begin{figure}[h]
    \centering
    \includegraphics[width=1\textwidth]{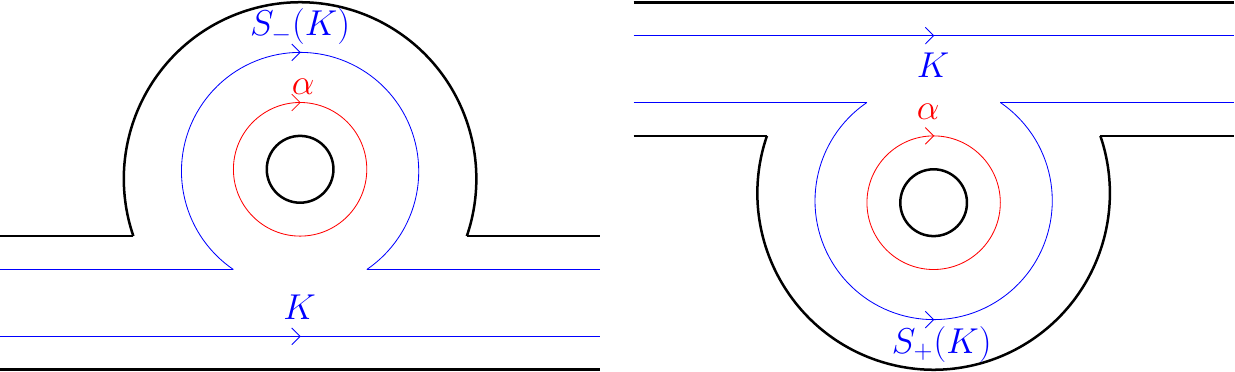}
    \caption{Negative and positive stabilization of a curve $K$ in an open book, respectively, drawn abstractly. One performs a right-handed Dehn twist along $\alpha.$}
    \label{fig::stab}
\end{figure}

\begin{prop}\label{surgerytwist}
Let $(\Sigma,\varphi)$ be an abstract open book supporting
$(M_{(\Sigma,\varphi)},\xi_{(\Sigma,\varphi)})$, and let
$\gamma\subset\Sigma$ be a simple closed curve.  Let $K_\gamma$ denote the
Legendrian realization of $\gamma$ on a page in $(M_{(\Sigma,\varphi)},\xi_{(\Sigma,\varphi)})$.  Then the contact $3$--manifold obtained from contact $(-1)$--surgery on $K_\gamma$ is supported by the abstract open book $(\Sigma,\varphi\circ D^+_\gamma)$.

\end{prop}

The following is the standard construction of planar open books supporting tight contact structures on lens spaces; see, for example, \cite{Sch07}.

In order to construct supporting planar open books, we begin with the annular open book $(\mathrm{A},D_c^+)$ supporting $(S^3,\xi_{\mathrm{std}})$, where $\mathrm{A}$ is the annulus and $c$ is the core curve of the annulus. The Legendrian realization of the curve $c$ is a Legendrian unknot with $\operatorname{tb}(K)=-1$ and $\operatorname{rot}(K)=0$. Iterating Proposition \ref{curveStab}, we can perform $a+b$ positive stabilizations, all preserving planarity, so that the resulting open book admits a Legendrian curve $K$ in the page $\Sigma\times\{0\}$ with $\operatorname{tb}(K)=-1-(a+b)$ and $\operatorname{rot}(K)=a-b$. See Figure \ref{fig::lensspaceOB}. By Proposition \ref{surgerytwist}, replacing the monodromy by $\varphi\circ D^+_{K}$ gives an open book supporting the contact manifold obtained from $(S^3,\xi_{\mathrm{std}})$ by contact $(-1)$-surgery along the Legendrian knot $K$. Hence the resulting planar open book has $2+a+b$ boundary components and supports the tight contact structure  $\xi^{2+a+b}_{a-b}$ on $L(2+a+b,1)$. See Figure \ref{fig::lensspaceOB2}. Therefore every tight contact structure $\xi^n_\rho$ on $L(n,1)$ has $\text{sg}(\xi)=0$ and binding number $\operatorname{bn}(\xi^n_\rho)\leq n$. This proves Lemma \ref{lemma::upper}.

\begin{figure}[h]
    \centering
    \includegraphics[width=1\textwidth]{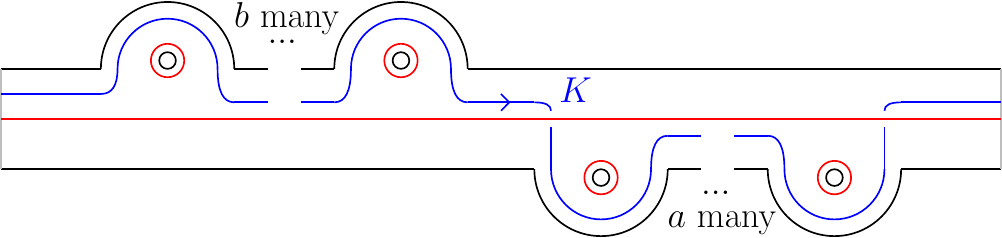}
    \caption{A planar open book for \((S^3,\xi_{\mathrm{std}})\) obtained
from the annular open book by \(a+b\) positive stabilizations. The red
curves are the right-handed Dehn twist curves. The blue curve \(K\)
satisfies \(tb(K)=-1-(a+b)\) and \(\operatorname{rot}(K)=a-b\).}
    \label{fig::lensspaceOB}
\end{figure}

\begin{figure}[h]
    \centering
    \includegraphics[width=0.8\textwidth]{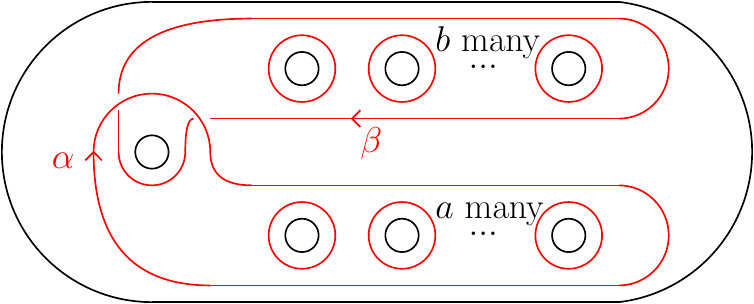}
    \caption{A planar supporting open book for the tight contact structures on $L(n,1)$. Here the monodromy is $D_\alpha^+\circ \varphi \circ D_\beta^+$, where $\varphi$ is the composition of the right-handed Dehn twists along the $a+b$ boundary components.}
    \label{fig::lensspaceOB2}
\end{figure}

\section{A Virtually Overtwisted Example}\label{sec::vot}

We now show that the equality \(bn(\xi^n_\rho)= n\) does not hold for all
tight contact structures on \(L(n,1)\). More precisely, we construct a planar open book with four binding components supporting $\xi^7_{-1}$; consequently, it also gives an upper bound for the binding number of $\xi^7_{+1}$, since these two contact structures are contactomorphic.

\begin{figure}[h]
    \centering
    \includegraphics[width=0.7\textwidth]{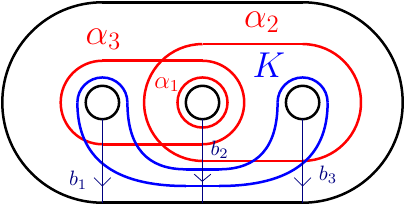}
    \caption{The four-holed sphere \(\Sigma\) with curves \(\alpha_1,\alpha_2,\alpha_3\), 
arc basis \(b_1,b_2,b_3\), and the curve \(K\). The open book
\((\Sigma,\varphi_0)\), where
\(\varphi_0=D^+_{\alpha_1}\circ D^+_{\alpha_3}\circ D^+_{\alpha_2}\) (read right to left),
supports \((S^3,\xi_{\mathrm{std}})\). The curve \(K\) is later used as the
surgery curve.}
    \label{fig::OB}
\end{figure}

Let \(\Sigma\) be the four-punctured sphere shown in Figure~\ref{fig::OB}, and let
\[
\varphi_0=D^+_{\alpha_1}\circ D^+_{\alpha_3}\circ D^+_{\alpha_2},
\] where the composition is read right to left.
The open book \((\Sigma,\varphi_0)\) supports \((S^3,\xi_{\mathrm{std}})\). Indeed, it is
obtained from the annular open book \((\operatorname{A},D_c^+)\) of
\((S^3,\xi_{\mathrm{std}})\) by two positive stabilizations. Let \(K\subset \Sigma\)
be the curve shown in Figure~\ref{fig::OB}. We first compute its classical
Legendrian invariants after Legendrian realization on the page.

\noindent\textbf{Claim.}
The Legendrian realization of \(K\) satisfies
\[
tb(K)=-6,\qquad \operatorname{rot}(K)=-1.
\]

\noindent\textit{Proof.}
We use the notation and formulas of Durst--Kegel recalled in Section~2.2, with
the arc basis shown in Figure~\ref{fig::OB}. For the factorization of
\(\varphi_0\), the matrix \(A\) is
\[
A=
\begin{pmatrix}
0&1&0\\
1&1&1\\
1&0&0
\end{pmatrix}.
\]
The corresponding vectors are
\[
\mathbf r=(0,0,0,1,1,0)^T,\qquad
\mathbf{l}=(-1,0,-1,-1,-1,0)^T.
\]
Solving \(Q\mathbf a=\mathbf l\) gives
\[
\mathbf a=(-3,2,-3,1,1,-2)^T.
\]
Therefore the Durst--Kegel formula for the Thurston--Bennequin invariant gives
\[
tb(K)
=
-\sum_{i=1}^3 (K\cdot b_i)^2-\langle \mathbf a, \mathbf l\rangle
=
-2-4
=
-6.
\]
Similarly,
\[
\operatorname{rot}(K)
=
r(K)-\langle \mathbf a, \mathbf r \rangle
=
1-2
=
-1.
\]
This proves the claim.
\hfill $\blacksquare$

It remains to identify the smooth knot type of \(K\).

\noindent\textbf{Claim.}
The knot \(K\) is smoothly isotopic to the unknot in \(S^3\).

\noindent\textit{Proof.}
Figure \ref{fig::surgerydiag} gives the Kirby diagram associated to the open book
\((\Sigma,\varphi_0)\), with the distinguished curve \(K\) drawn in blue.
Figures \ref{fig:kirby-unknot-first} and \ref{fig:kirby-unknot-second} give an explicit sequence of Kirby moves reducing the diagram of Figure \ref{fig::surgerydiag} to the empty surgery diagram of $S^3$, while carrying the blue curve $K$ along throughout the sequence. In the final panel of Figure \ref{fig:kirby-unknot-second}, no framed surgery components remain and the image of $K$ is visibly the standard unknot.
\hfill $\blacksquare$

\begin{figure}[h]
    \centering
    \includegraphics[width=0.6\textwidth]{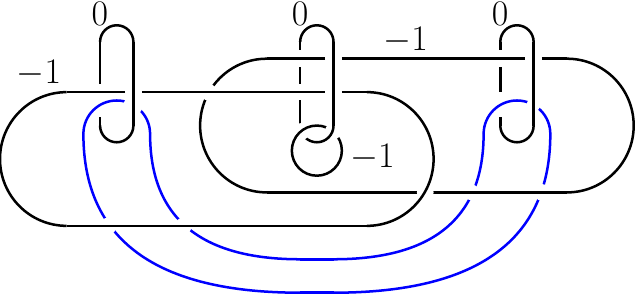}
    \caption{Surgery diagram for the manifold described by the open book decomposition $(\Sigma,\varphi_0)$. The distinguished blue curve is $K$.}
    \label{fig::surgerydiag}
\end{figure}

We can now finish the construction. Since \(K\) is a Legendrian unknot with
\(tb(K)=-6\) and \(\operatorname{rot}(K)=-1\), contact \((-1)\)-surgery along \(K\)
yields $(L(7,1),\xi^7_{-1}).$

Moreover, \(K\) lies on a page of the open book \((\Sigma,\varphi_0)\) supporting
\((S^3,\xi_{\mathrm{std}})\). By Proposition \ref{surgerytwist}, contact \((-1)\)-surgery along
\(K\) is supported by the open book
\(
(\Sigma,\varphi_0\circ D_K^+).
\)
Thus \((\Sigma,\varphi_0\circ D_K^+)\) is a planar open book with four binding
components supporting the tight contact $3-$manifold $(L(7,1),\xi^7_{-1}).$

\section{The lower bound}\label{sec::lower}

We now turn to the proof of Theorems \ref{thm::bn} and \ref{thm::bn56}.

\begin{proof}[Proof of Theorem~\ref{thm::bn}]
Let $(\Sigma,\varphi)$ be an abstract planar open book supporting the tight contact manifold $(L(n,1),\xi^n_\rho)$. By Lemma \ref{lemma:factorization} we may write $\varphi=D_{\alpha_d}^+ \circ \cdots \circ D_{\alpha_1}^+$ where each $\alpha_i$ is a closed curve in $\Sigma$, and $d$ is the number of boundary components of $\Sigma$. Choose an outer boundary component of $\Sigma$ and label the remaining
$d-1$ boundary components as inner boundary components, as in Figure \ref{fig::stdOB}. With the notation from Section \ref{sec::prelim}, we get the basis $\{\beta_1,\dots,\beta_{d-1}\}$ of $H_1(\Sigma)$.

Define the binary $(d-1)\times d$ matrix $A=(A_{ij})$ by $A_{ij}:=\alpha_j\cdot b_i$. Equivalently, the $j$-th column of $A$ records which inner boundary components are enclosed by $\alpha_j$. Thus, in $H_1(\Sigma)$, we have $\alpha_j=\sum_{k=1}^{d-1} A_{kj}\beta_k$. Thus $A_{kj}\in\{0,1\}$.\\

\noindent\textbf{Claim.}
    $\det (AA^T)=n$.

\noindent\textit{Proof.}
We use the standard computation of the first homology of a closed $3-$manifold from an abstract open book. With respect to the above choice of arcs, one has $ H_1(M_{(\Sigma,\varphi)}) \cong \langle  \beta_1,\cdots,\beta_{d-1} | \varphi_*(b_1)-b_1,\cdots,\varphi_*(b_{d-1})-b_{d-1} \rangle $ (see \cite{etnyre2006invariantscontactstructuresopen}). Here $\varphi_*(b_i)-b_i$ is naturally viewed as a homology class in  $H_1(\Sigma)$.

A positive Dehn twist along a curve $\alpha$ changes the relative homology class of an arc $b_i$ by $(D_\alpha^+)_*(b_i)=b_i+(\alpha\cdot b_i)\alpha.$ Since $\Sigma$ is planar, the algebraic intersection number of any two closed curves in $\Sigma$ is zero. Therefore the contributions of the Dehn twists simply add, and we obtain $\varphi_*(b_i)=b_i+\sum_{j=1}^d (\alpha_j\cdot b_i)\alpha_j .$ Using the expression of each $\alpha_j$ in the basis $\{\beta_1,\dots,\beta_{d-1}\}$, we get $$\varphi_*(b_i)-b_i=\sum_{j=1}^d A_{ij}\left(\sum_{k=1}^{d-1} A_{kj}\beta_k\right)=\sum_{k=1}^{d-1}\left(\sum_{j=1}^d A_{ij}A_{kj} \right)\beta_k=\sum_{k=1}^{d-1}(AA^T)_{ik}\beta_k.$$ Therefore the relations defining $H_1(M_{(\Sigma,\varphi)})$ are given precisely by the rows of the matrix $AA^T$. Since $AA^T$ is symmetric, the subgroup generated by these $d-1$ vectors is $\text{Im}(AA^T)$. Then $H_1(M_{(\Sigma,\varphi)})\cong \mathbb Z^{d-1}/\operatorname{Im}(AA^T).$ Since $(\Sigma,\varphi)$ supports $L(n,1)$, we have $H_1(M_{(\Sigma,\varphi)}) \cong H_1(L(n,1);\mathbb Z) \cong \mathbb Z/n\mathbb Z.$ In particular this group is finite of order $n$. Therefore $|\det(AA^T)|=n$ (see \cite[Ch.XIV, Ex.16]{LangAlgebra}) and $AA^T$ has full rank. Since $AA^T$ is positive definite, its determinant is positive. Thus $\det(AA^T)=n.$
This proves the claim.

\hfill $\blacksquare$

Following Section \ref{sec:invariants} we define the $(2d-1)\times (2d-1)$ matrix 
$$Q=\begin{pmatrix}0_{(d-1)\times (d-1)} & -A \\-A^T &  -A^TA-I_{d}\end{pmatrix}.$$ We now solve the system $$Q\begin{pmatrix}y\\x\end{pmatrix}=\begin{pmatrix}0\\r\end{pmatrix}$$ for $x\in\mathbb{Q}^{d},y\in\mathbb{Q}^{d-1}$ and $r\in\mathbb{Z}^d$ defined as $r_i=r(\alpha_i)$. In the notation of Section \ref{sec:invariants}, $(y,x)^T=\mathbf{b},(0,r)^T=\mathbf{r}$. The two equations we get are $Ax=0, r=-A^Ty-x$. Taking the dot product with $x$ gives $x\cdot r=-x\cdot A^Ty-x\cdot x=-x\cdot x.$ 

Consider the vector $v\in\mathbb{Z}^d$ given by $v_i=(-1)^i\det(A_i)$, where $A_i$ is the $(d-1)\times (d-1)$ matrix obtained from $A$ by deleting the $i$-th column. \\

\noindent\textbf{Claim.} $Av=0$

\noindent\textit{Proof.}
    Let $\text{a}_i\in \mathbb{Z}^d$ be the $i$-th row of $A$. $Av=0$ is equivalent to showing that $\text{row}_k\cdot v=0$ for each $k=1,\cdots,d-1$. Notice $\text{a}_k\cdot v=\sum_{i=1}^d A_{ki}(-1)^i\det(A_i)$. Consider the matrix $$\begin{pmatrix}\text{a}_k\\\text{a}_1\\ \vdots \\ \text{a}_{d-1}\end{pmatrix},$$ which has determinant zero. Expanding the determinant along the first row gives  $$\sum_{i=1}^d(-1)^{i+1}A_{ki}\det(A_i),$$ which proves our claim.
    
\hfill $\blacksquare$

The Cauchy–Binet formula \cite[Eq.~(3.14), p.~298]{TaoRMT} gives $n=\det(AA^T)=\sum_{i=1}^d \det(A_i)^2=\sum_{i=1}^d v_i^2$. Since $\ker A$ is one dimensional and $x\in\ker A$, there is a scalar $\lambda\in \mathbb{Q}$ such that $x=\lambda v$. 
By Lemma \ref{lemma::r}, each entry satisfies $r(\alpha_i)=|{\text{holes enclosed by }\alpha_i}|-1$, and therefore $r=A^T\mathbf{1}_{d-1}-\mathbf{1}_d$, where $\mathbf{1}_k$ represents the all-ones vector in $\mathbb{Z}^k$. To determine $\lambda$, take the dot product with $v$; $v\cdot r=v^T(-A^Ty-x)=-v\cdot x=-\lambda v\cdot v$ and hence $$\lambda=-\frac{v\cdot r}{v\cdot v}=-\frac{v^T(A^T\mathbf{1}_{d-1}-\mathbf{1}_d)}{n}=-\frac{-\sum_i v_i}{n}=\frac{S}{n},$$ where $S:=\sum_i v_i\in \mathbb{Z}$.

By the Durst--Kegel formula for the $d_3$ invariant one gets $$d_3(\xi)=-1+\frac{1}{4}x\cdot r-\frac{3}{4}\sigma(Q).$$

\noindent\textbf{Claim.}
$\sigma(Q)=-1$

\noindent\textit{Proof.}

Let $$P=\begin{pmatrix}0_{(d-1)\times d} & I_{d-1}\\I_d & 0_{d\times(d-1)}\end{pmatrix}.$$ Since $P$ is a permutation matrix and hence an orthogonal matrix, \(Q\) and \(P^TQP\) are congruent, so they have the same signature by Sylvester's law of inertia. A direct computation gives $$P^T Q P=\begin{pmatrix}-A^TA-I_d & -A^T\\ -A & 0_{(d-1)\times(d-1)}\end{pmatrix}.$$
Applying the Haynsworth inertia additivity formula \cite{HaynsworthOstrowski1968}, $$\sigma(P^TQP)=\sigma(-A^TA-I_d)+\sigma(A(A^TA+I_d)^{-1}A^T).$$
Since $-A^TA-I_d$ is negative definite, we have $\sigma(-A^TA-I_d)=-d$. If $z\in \mathbb{R}^{d-1}\setminus\{0\}$, then $z^TA(A^TA+I_d)^{-1}A^Tz=(A^Tz)^T(A^TA+I_d)^{-1}(A^Tz)>0$, since $(A^TA+I_d)^{-1}$ is positive definite and $A^Tz$ is nonzero since $A^T$ is injective (from $\det (AA^T)=n\neq0$).
Therefore $\sigma(Q)=-d+(d-1)=-1$.

\hfill $\blacksquare$

It follows that  
\begin{align*}
    d_3(\xi)&=-\frac{1}{4}(1-x\cdot r)=-\frac{1}{4}\left(1+\sum_ix_i^2\right)\\
    &=-\frac{1}{4}\left(1+\sum_i \left( \frac{S}{n}v_i\right)^2\right)=-\frac{1}{4}\left(1+\frac{S^2}{n^2}\sum_i v_i^2\right)\\
    &=-\frac{1}{4}\left(1+\frac{S^2}{n^2}n\right)=-\frac{1}{4}\left(1+\frac{S^2}{n}\right)
\end{align*}
We have seen in Section \ref{sec::lensspaces} that $d_3(\xi)=-\frac{1}{4}\left(1+\frac{\rho^2}{n}\right)$. Comparing the two expressions for $d_3(\xi)$ gives the identities
\begin{align*}
    |\sum_{i=1}^d v_i|&=|S|=|\rho|,\qquad \sum_{i=1}^dv_i^2=n
\end{align*}

We now use only the two identities above. Replacing $v$ by $-v$ if necessary, assume $S\ge 0$. Since every column of $A$ is nonzero and all entries of $A$ are nonnegative and $Av=0$, it follows that $v$ must have at least one positive and at least one negative entry. Let $m$ be the number of non-zero entries in $v$, hence $m\leq d$. Define $$\mathcal{P}:=\sum_{v_i>0}v_i, \qquad \mathcal{N}:=-\sum_{v_i<0}v_i.$$ Then $\mathcal{P}-\mathcal{N}=S$. Since $v$ has a negative entry, it follows that $\mathcal{N}\geq 1$ and hence $\mathcal{P}\geq S+1$. We have $$n=\sum_{v_i>0}v_i^2+\sum_{v_i<0}v_i^2\geq \sum_{v_i>0}v_i^2+1.$$ By Cauchy-Schwarz $$(S+1)^2\leq \mathcal{P}^2\leq \#\{i:v_i>0\}\sum_{v_i>0}v_i^2\leq (m-1)(n-1)\leq (d-1)(n-1)$$
and hence $$d \geq \left\lceil 1+\frac{(S+1)^2}{n-1}\right\rceil=\left\lceil 1+\frac{(|\rho|+1)^2}{n-1}\right\rceil$$

\end{proof}

\begin{proof}[Proof of Theorem \ref{thm::bn56}]
\leavevmode\par

\textbf{Case $n=2$.} Alexander's trick shows that the mapping class group of the disk is trivial, hence the only $3-$manifold supported by an open book with disk pages can be $S^3$. It follows that $\operatorname{bn}(\xi^2_\rho)\geq2$. The reverse inequality was established in Section \ref{sec::upper}.

\textbf{Case $n=3$.} As in the previous case, we have $\operatorname{bn}(\xi^3_\rho)\geq2$, while we showed earlier that
$\operatorname{bn}(\xi^3_\rho)\leq 3$. By Lemma \ref{lemma:factorization}, a supporting planar open book with two boundary components would have to have monodromy equal to a product of two right-handed Dehn twists. Since there is a unique closed nontrivial curve in the annulus, namely the core, the monodromy can only be $D^+_c\circ D^+_c$, where $c$ denotes the core of the annulus. But we have seen in Section \ref{sec::upper} that this supports the tight $L(2,1)$. Hence $\operatorname{bn}(\xi^3_\rho)=3$ follows. 

\textbf{Case $n=4$.}  
Assume $\xi^4_\rho$ admits a planar open book with $d$ boundary components. By Lemma \ref{lemma:factorization} we know that the monodromy may be factorized into either $d-1$ or $d$ right-handed Dehn twists. If $d=2$ then this is either the annular open book for the tight $S^3$ or the open book for the tight $L(2,1)$. Hence only the case $d=3$ remains. Here we adapt the argument in the proof of \cite[Theorem 5.2]{etnyre2006invariantscontactstructuresopen}. Assume we have a supporting planar open book with $3$ boundary components. Since the page is a pair of pants, every essential simple closed curve is boundary-parallel. Hence every nontrivial positive Dehn twist is isotopic to a twist about one of the three curves $\gamma_1,\gamma_2,\gamma_3$ shown in Figure \ref{fig::3bnds} and the supporting open book must have monodromy $D_{\gamma_3}^{n_3}\circ D_{\gamma_2}^{n_2}\circ D_{\gamma_1}^{n_1}$ where $n_1,n_2,n_3\in \mathbb{Z}_{\geq 0}$ and $n_1+n_2+n_3$ is equal to either $2$ or $3$.

Let $A$ be the integer matrix associated to this factorization in the
sense of Durst--Kegel. Its columns consist of $n_1$ copies of
\(\binom{1}{0}\), \(n_2\) copies of \(\binom{0}{1}\), and \(n_3\) copies of
\(\binom{1}{1}\). One may compute that \[
AA^T
=
\begin{pmatrix}
n_1+n_3 & n_3\\
n_3 & n_2+n_3
\end{pmatrix}
\]

and hence $\det (AA^T)=n_1n_2+n_1n_3+n_2n_3$. As observed in the proof of Theorem \ref{thm::bn},  we must have $\det (AA^T)=|H_1(L(4,1))|=4$. However, when $n_1+n_2+n_3\in \{2,3\}$, one has $n_1n_2+n_1n_3+n_2n_3\leq 3$, which is a contradiction. Thus, no supporting planar open book with two
or three binding components exists. Together with the previously established
upper bound, this proves that
$\operatorname{bn}(\xi_\rho^{4})=4.
$

\begin{figure}[h]
    \centering
    \includegraphics[width=0.6\textwidth]{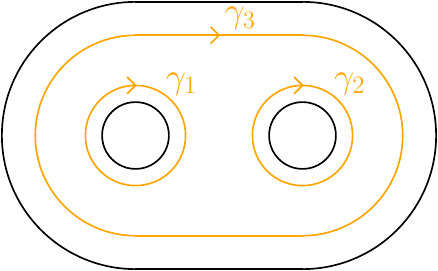}
    \caption{A planar open book with $3$ boundary components must have monodromy $D_{\gamma_3}^{n_3}\circ D_{\gamma_2}^{n_2}\circ D_{\gamma_1}^{n_1}$ for some $n_1,n_2,n_3\in \mathbb{Z}$.}
    \label{fig::3bnds}
\end{figure}

We now turn to the cases where $n \geq 5$. Throughout, let \((\Sigma,\varphi)\) be a planar open book supporting a tight contact structure on \(L(n,1)\) with \(d\) binding components. Define $A\in \{0,1\}^{(d-1)\times d}$ and $v\in\mathbb{Z}^d$ as in the proof of Theorem \ref{thm::bn}. Recall that we have \[Av=0,\qquad \sum_{i=1}^d v_i^2=n, \qquad \sum_{i=1}^d v_i=|\rho|, \qquad \det(AA^T)=n.\] We use these four identities to prove the rest of the cases.

\noindent\textbf{Case $n=5$.}
Suppose, for contradiction, that \(d\leq 4\). Since $\sum_{i=1}^d v_i^2=5,$ the nonzero entries of \(v\), after multiplying \(v\) by \(-1\) and permuting the columns, must be either $(2,1)$ or $(2,-1).$ Let \(C_j\) denote the \(j\)-th column of \(A\). In the first case, the equation \(Av=0\) gives $2C_1+C_2=0.$ Since all entries of \(C_1\) and \(C_2\) are nonnegative, this forces
\(C_1=C_2=0\), contradicting the fact that every column of \(A\) is nonzero. In the second case, the equation \(Av=0\) gives $2C_1-C_2=0$. Since \(C_1\) and \(C_2\) are binary columns, this
is possible only if \(C_1=C_2=0\), again contradicting the nonzero-column condition.

Therefore no planar supporting open book for a tight contact structure on
\(L(5,1)\) can have \(d\leq 4\) binding components. Hence $\operatorname{bn}(\xi^5_\rho)=5$ follows.

\noindent\textbf{Case $n=6$.}
Suppose, for contradiction, that \(d\leq 5\). Since $\sum_{i=1}^d v_i^2=6,$
and since \(d\leq 5\), the vector \(v\) cannot have six nonzero entries all equal
to \(\pm 1\). Hence, after multiplying by \(-1\) and permuting the columns, the
nonzero entries of \(v\) must be one of
\[
        (2,1,1), \qquad (2,1,-1), \qquad (2,-1,-1).
\]
In the first case, \(Av=0\) gives $2C_1+C_2+C_3=0.$ Since all entries are nonnegative, this forces \(C_1=C_2=C_3=0\), contradicting the nonzero-column condition.

In the second case, \(Av=0\) gives $2C_1+C_2-C_3=0.$ This is also impossible, since then $C_3$ would necessarily have an entry larger than $1$. It remains to consider the third case. Then \(Av=0\) gives $2C_1-C_2-C_3=0.$ Since all entries are $0$ or $1$, it would follow that $C_1=C_2=C_3.$ Thus the first three columns span at most a one-dimensional subspace. Together with the remaining $d-3$ columns, they span a subspace of dimension at most $d-2$. But from $\det(AA^T)=n\neq 0$ we see that $A$ has full row rank and hence \(\operatorname{rank}(A)=d-1\), which is a contradiction.

Therefore no planar supporting open book for a tight contact structure on
\(L(6,1)\) can have \(d\leq 5\) binding components. Hence we again have $\operatorname{bn}(\xi^6_\rho)=6$.

\noindent\textbf{Case $n\geq 7$.} For $\rho=\pm (n-2)$, we are finished by Corollary \ref{cor::untight}. Assume next that \(|\rho|=n-4\). Suppose, for contradiction, that
\(d\le n-1\). After replacing \(v\) by \(-v\), if necessary, we may
assume that \(\sum_i v_i=n-4\). Hence $\sum_i (v_i^2-v_i)=4 .$
Since \(v_i^2-v_i\) is equal to \(0\) for \(v_i=0,1\), equal to \(2\)
for \(v_i=2,-1\), and at least \(6\) for any other value of $v_i$, the nonzero
entries of \(v\), up to permutation and omitting zero entries, must be one
of
\[
        (2,2,\underbrace{1,\ldots,1}_{n-8}),
        \qquad
        (2,\underbrace{1,\ldots,1}_{n-5},-1),
        \qquad
        (\underbrace{1,\ldots,1}_{n-2},-1,-1),
\]
where the first possibility occurs only when \(n\ge 8\).

The first and second possibilities are impossible because $Av=0$ would yield the equation 
$2C_1+2C_2+C_3+\cdots+C_{n-6}=0$ or $2C_1+C_2+\cdots+C_{n-4}-C_{n-3}=0$, both of which are impossible since the entries of $C_i$ are $0$ or $1$ and all $C_i$ are nonzero. The third possibility contradicts $d<n$. Since the upper bound \(d=n\) was already constructed, we get
\[\operatorname{bn}(\xi_\rho^n)=n\qquad\text{for }|\rho|=n-4.\]

It remains to consider \(|\rho|=n-6\). We prove that \(d\ge n-3\). Suppose,
for contradiction, that \(d\le n-4\). Again assume
\(\sum_i v_i=n-6\). Then $\sum_i (v_i^2-v_i)=6 .$

Thus, up to permutation and omitting zero entries, the nonzero entries of
\(v\) must be one of
\[
\begin{array}{lll}
\bigl(3,\underbrace{1,\ldots,1}_{n-9\ \mathrm{times}}\bigr),&
\bigl(-2,\underbrace{1,\ldots,1}_{n-4\ \mathrm{times}}\bigr),&
\bigl(2,2,2,\underbrace{1,\ldots,1}_{n-12\ \mathrm{times}}\bigr),\\[0.8em]
\bigl(2,2,\underbrace{1,\ldots,1}_{n-9\ \mathrm{times}},-1\bigr),&
\bigl(2,\underbrace{1,\ldots,1}_{n-6\ \mathrm{times}},-1,-1\bigr),&
\bigl(\underbrace{1,\ldots,1}_{n-3\ \mathrm{times}},-1,-1,-1\bigr).
\end{array}
\]

Again, some of these only make sense when $n$ is large enough. The second, fifth, and sixth possibilities have respectively \(n-3\), \(n-3\), and \(n\) nonzero entries, contradicting \(d\le n-4\). The first, third and fourth possibilities are impossible because the equations
\begin{align*}
        3C_1+C_2+C_3+\cdots+C_{n-8}&=0,\\
        2C_1+2C_2+2C_3+C_4+\cdots+C_{n-9}&=0,\\
        2C_1+2C_2+C_3+\cdots+C_{n-7}-C_{n-6}&=0
\end{align*}
cannot be solved for nonzero vectors $C_i$ with entries $0$ or $1$.

Therefore no planar open book supporting \(\xi_\rho^n\) can have
\(d\le n-4\) binding components. Hence \(d\ge n-3\). 

We now show the existence of a supporting open book with $d=n-3$. Consider the supporting planar open book for the tight $S^3$ as in Section \ref{sec::vot}, which has $4$ boundary components and contains a Legendrian unknot $K$ in its
 page with $\operatorname{tb}(K)=-6$ and $\operatorname{rot}(K)=-1$. We may now positively stabilize this open book $n-7$ times as in Proposition \ref{curveStab} such that the resulting open book admits a Legendrian unknot $K'$ in its page with $\operatorname{tb}(K')=-(n-1)$ and $\operatorname{rot}(K')=-(n-6)$. Now performing a contact $(-1)$-surgery along $K'$ yields $(L(n,1),\xi^n_{-(n-6)})$ with a supporting planar open book with $4+n-7=n-3$ binding components. Since $(L(n,1),\xi^n_{-(n-6)})$ and $(L(n,1),\xi^n_{+(n-6)})$ are contactomorphic (as explained in Section \ref{sec::lensspaces}), we have shown $$\operatorname{bn}(\xi_\rho^n)=n-3 \qquad\text{for }|\rho|=n-6.$$

\end{proof}

\appendix

\section{Kirby-calculus verification that \(K\) is an unknot}

In this appendix we record the Kirby moves used in Section \ref{sec::vot} to verify that the curve \(K\) in Figure \ref{fig::surgerydiag} is smoothly isotopic to the unknot. Throughout the sequence, the blue curve denotes the image of \(K\). The framed components are modified by isotopy, handle slides, blow-downs, and cancellations, while the blue curve is carried along by the induced diffeomorphisms of the surgered $3$-manifold. The first part of the reduction is shown in Figure \ref{fig:kirby-unknot-first}, and the remaining moves are shown in Figure \ref{fig:kirby-unknot-second}.

\newpage

\begin{figure}[h]
    \centering
    \includegraphics[width=0.85\textwidth]{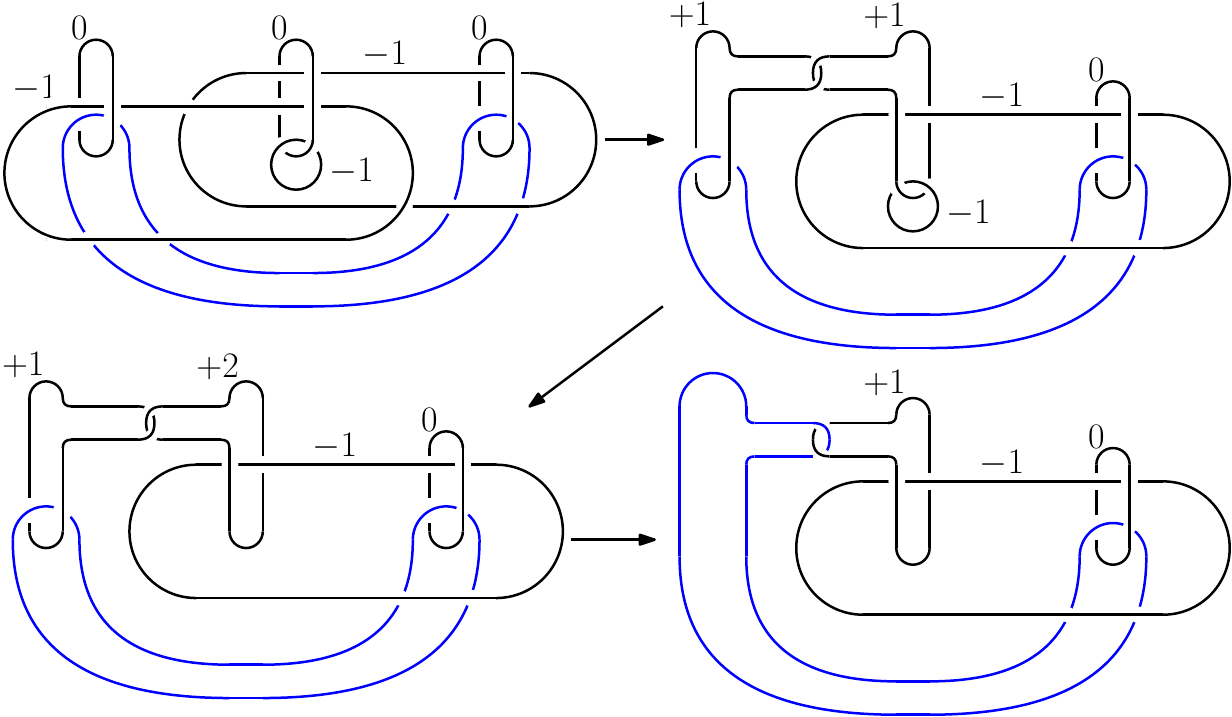}
    \caption{The first part of the Kirby-calculus reduction of the diagram in Figure \ref{fig::surgerydiag}. The first move is a blow-down of the leftmost $-1$-framed component, the second is a blow-down of the inner $-1$-framed component, and the third is a blow-down of the leftmost $+1$-framed component.}
    \label{fig:kirby-unknot-first}
\end{figure}
\begin{figure}[h]
    \centering
    \includegraphics[width=0.8\textwidth]{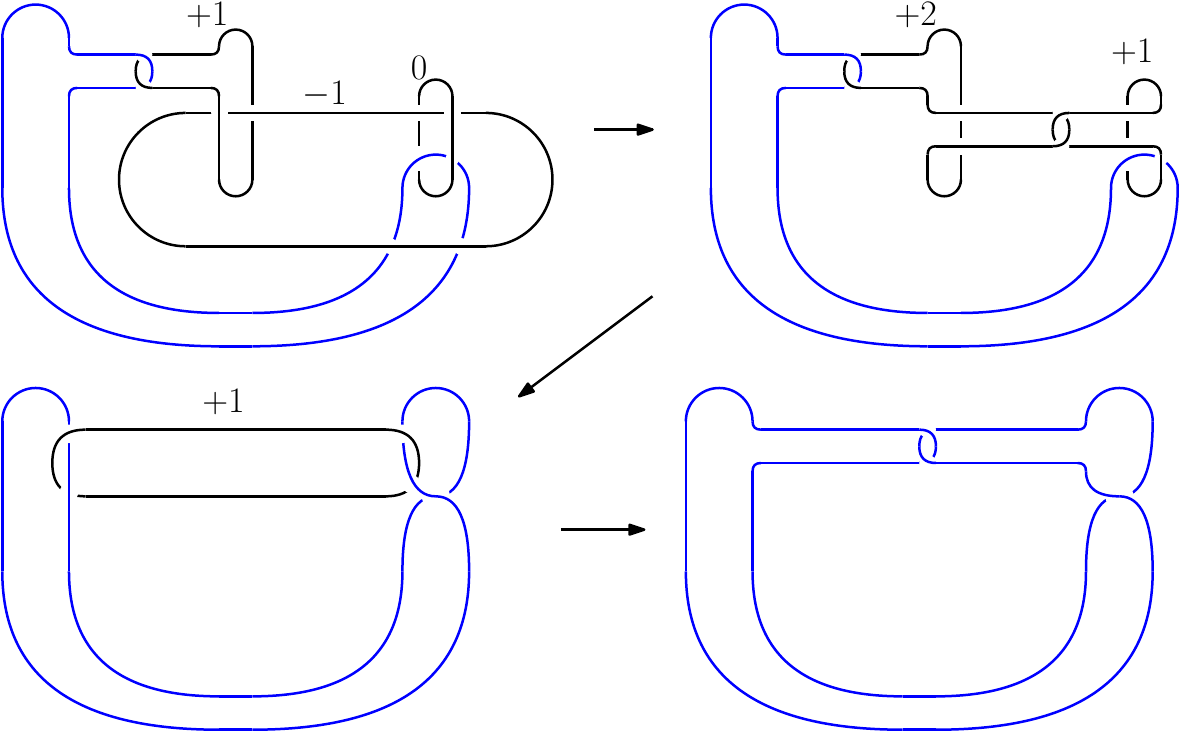}
    \caption{Continuation of the sequence from Figure~\ref{fig:kirby-unknot-first}. The first move is a blow-down of the remaining $-1$-framed component, the second is a blow-down of the remaining $+1$-framed component, and the third is a blow-down of the remaining $+1$-framed component.
     The image of \(K\) is visibly the standard unknot.}
    \label{fig:kirby-unknot-second}
\end{figure}

\FloatBarrier

\bibliographystyle{alpha}
\bibliography{sample}
\end{document}